\documentclass[12pt, a4paper, parskip=half, abstracton]{scrartcl}

\usepackage{array}
\usepackage{marginnote}
\usepackage{xcolor}
\usepackage{amscd,amssymb,amsfonts,amsmath,latexsym,amsthm}
\usepackage{hyperref}
\usepackage[all,cmtip]{xy}
\usepackage{mathrsfs}
\usepackage{graphicx}
\usepackage{bm} 

\usepackage{etoolbox}

\def\hB{\hspace*{\fill}$\qed$}

\usepackage{chngcntr} 
\usepackage[nottoc]{tocbibind} 
\usepackage{defs_pp1}
\usepackage{slashed}
\usepackage[utf8]{inputenc}
\usepackage{microtype}
\usepackage[english]{babel}

\title{Controlled objects as a symmetric monoidal functor}
\author{
Ulrich Bunke\thanks{Fakult{\"a}t f{\"u}r Mathematik,
Universit{\"a}t Regensburg,
93040 Regensburg,
GERMANY\newline
ulrich.bunke@mathematik.uni-regensburg.de} 
 and
Luigi Caputi\thanks{Fakult{\"a}t f{\"u}r Mathematik,
	Universit{\"a}t Regensburg,
	93040 Regensburg,
	GERMANY\newline
	luigi.caputi@mathematik.uni-regensburg.de}}

\numberwithin{equation}{section}
\counterwithout{footnote}{section}

\newtheorem{theorem}{Theorem}[section] 
\newtheorem{prop}[theorem]{Proposition}
\newtheorem{lem}[theorem]{Lemma}

\newtheorem{ddd}[theorem]{Definition}
\newtheorem{kor}[theorem]{Corollary}
\newtheorem{ass}[theorem]{Assumption}

\theoremstyle{remark}
\theoremstyle{definition}

\newtheorem{rem}[theorem]{Remark}
\newtheorem{conv}[theorem]{Convention}

\newcommand{\UK}{\mathrm{UK}}

\newcommand{\Mor}{\mathrm{Mor}}

\newcommand{\bQ}{\mathbf{Q}}

\newcommand{\BC}{\mathbf{BornCoarse}}

\newcommand{\Fin}{\mathbf{Fin}}

\newcommand{\bB}{{\mathbf{B}}}

\newcommand{\cP}{\mathcal{P}}

\newcommand{\cK}{\mathcal{K}}

\newcommand{\Add}{{\mathtt{Add}}}

\newcommand{\bA}{{\mathbf{A}}}

\newcommand{\Alg}{{\mathbf{Alg}}}

\newcommand{\cU}{{\mathcal{U}}}

\newcommand{\cD}{{\mathcal{D}}}

 \newcommand{\Cat}{{\mathbf{Cat}}}

\newcommand{\Born}{\mathbf{Born}}
\newcommand{\Coarse}{\mathbf{Coarse}}

\renewcommand{\Add}{\mathbf{Add}}

\begin{document}
	
\maketitle

\begin{abstract}
The goal of this paper is to associate functorially to every symmetric monoidal additive category $\bA$ with a strict $G$-action a  {lax} symmetric monoidal functor
$\bV_{\bA}^{G}:G\BC\to \Add_{\infty}$ from the  symmetric monoidal category of $G$-bornological coarse spaces $G\BC$ to the symmetric monoidal {$\infty$-}category of additive categories $\Add_{\infty}$. This allows to 
{refine} equivariant coarse algebraic $K$-homology   to a {lax} symmetric monoidal functor.  \end{abstract}
\tableofcontents

\section{Introduction}

A  $\cC$-valued equivariant coarse homology theory is a functor $$E:G\BC\to \cC$$ satisfying a certain family of axioms  {\cite[Def. 3.10]{equicoarse}}.  Here, $G\BC$ is the category of $G$-bornological coarse spaces and $\cC$ is a stable cocomplete $\infty$-category.
We refer to \cite[Sec. 2.1]{equicoarse},  or {Section~\ref{geriogjheroigfwefwefw},} for details. The category $G\BC$ has a symmetric monoidal structure
$\otimes$, and if {also} $\cC$ has a symmetric monoidal structure, then we can ask whether {the functor} $E$ can be refined to a
lax symmetric monoidal functor.   Such a refinement can simplify calculations or can be applied to {obtain} localization results, see  \cite{bulu}.

In the present paper{, as an example for  $E$,} we consider the universal coarse algebraic $K$-homology  $$\UK\cX_{\bA}^{G}:G\BC\to \cM_{loc}$$ associated to an additive category $\bA$ with a strict action of the group $G$, where $\cM_{loc}$ is the {stable $\infty$-category on non-commutative motives, defined as the target of the} universal localizing invariant {$\cU_{loc}$} of Blumberg-Gepner-Tabuada \cite{MR3070515}. The functor $\UK\cX_{\bA}^{G}$
has been introduced in \cite{buci} as the universal variant of the spectrum-valued 
coarse algebraic $K$-homology  $K\cX_{\bA}^{G}$ constructed in \cite[Ch. 8]{equicoarse}.

By \cite[Thm. 5.8]{bgt-2}, the {$\infty$-}category $\cM_{loc}$ has a symmetric monoidal structure.

The main result of the present paper is the following theorem.
\begin{theorem}\label{rgreioghjo34tergergeg} A symmetric monoidal structure on  $\bA$  induces a  {lax} symmetric monoidal refinement
of the  functor $\UK\cX_{\bA}^{G}$.
\end{theorem}

The functor $\UK\cX_{\bA}^{G}$ is constructed as a composition of
functors
$$\UK\cX_{\bA}^{G}:G\BC\stackrel{\bV_{\bA}^{G}}{\to} \Add_{1}\stackrel{\UK}{\to} \cM_{loc}\ ,$$
where $X\mapsto \bV_{\bA}^{G}(X)$ associates to a $G$-bornological coarse space $X$ its additive category of  equivariant $X$-controlled $\bA$-objects {(see Definition \ref{def:Xcontrolledobject})}, and $\bC\mapsto \UK(\bC)$ sends an additive category $\bC$ to
the motive $\cU_{loc}(\Ch^{b}(\bC)[W^{-1}])$ of the associated stable $\infty$-category 
$\Ch^{b}(\bC)[W^{-1}]$ of bounded chain complexes over $\bC$, where $W$ is the set  of homotopy equivalences.
Since $\UK$ sends equivalences of additive categories to equivalences it has a factorization
$$\UK:\Add_{1}\stackrel{loc}{\to} \Add_{\infty}\stackrel{\UK_{\infty}}{\to} \cM_{loc} \ ,$$
where $loc:\Add_{1}\to \Add_{\infty}:=\Add_{1}[W^{-1}_{\Add}]$ is the localization at the equivalences {$W_{\Add}$ of additive categories,} in the realm of $\infty$-categories.

Theorem \ref{rgreioghjo34tergergeg} follows from the following two assertions:

\begin{theorem}[{Theorem \ref{mainthm}}]\label{rgreioghjo34tergergeg1}
A symmetric monoidal structure on $\bA$  induces a {lax} symmetric monoidal refinement
$$\bV_{\bA}^{G,\otimes}:G\BC\to \Add_{\infty}^{\otimes}$$
of the functor
$loc\circ \bV_{\bA}^{G}$.
\end{theorem}

\begin{theorem}[Theorem \ref{rgreioghjo34tergergeg21}] \label{rgreioghjo34tergergeg2}
The functor $\UK_{\infty}$ admits a  {lax} symmetric monoidal refinement
$$\UK_{\infty}^{\otimes}:\Add_{\infty}^{\otimes}\to \cM_{loc}^{\otimes}\ .$$
\end{theorem}

The main difficulty in {proving} Theorem \ref{rgreioghjo34tergergeg1} is that the symmetric monoidal category of {small} additive categories is of a $2$-categorical nature. A {pedestrian} approach to {the proof of} this theorem would thus require to work with symmetric monoidal structures on $2$-categories and therefore   tedious considerations of a large set of commuting diagrams.  In this paper we prefer to use the language of symmetric monoidal $\infty$-categories. 
In Section \ref{vevgebvveerverv}{, by using  the {Grothendieck} construction,}
 we encode the functor $\bV_{\bA}^{G}{\colon G\BC \to \Add_{1}}$ into {a} cocartesian fibration $\cV_{\bA}^{G}\to G\BC$
coming from an op-fibration of $1$-categories. We then encode a symmetric monoidal
refinement of the functor $\bV_{\bA}^{G}$ into a symmetric monoidal  structure on $\cV_{\bA}^{G}$ and a 
symmetric monoidal refinement of  the functor to $G\BC$. This only requires $1$-categorical considerations. 
The machine of $\infty$-categories then produces, as explained in Section \ref{fgoiwfwefewfwefwefwe},  the asserted  symmetric monoidal refinement  
in Theorem \ref{rgreioghjo34tergergeg1}.

The technical  results  Theorem  \ref{gui34tefrgeg} and Theorem  \ref{rgio34t34r34rerg} might be of independent interest in cases where one wants to construct symmetric monoidal refinements of functors from $1$-categories to $\Cat_{1}$ or $\Add_{1}$.

Theorem \ref{rgreioghjo34tergergeg2} is shown in Section \ref{foiwejfowfwefefwef} by combining   various results  in the literature on $dg$-categories.

 {\em Acknowledgements: We thank Denis-Charles Cisinksi and Thomas Nikolaus for helpful discussion. U.B. was supported by the SFB 1085 (Higher Invariants) and L.C. was {supported} by the GK 1692 (Curvature, Cycles, and Cohomology).}

\section{Symmetric monoidal functors to $\Cat$ and $\Add$}\label{fgoiwfwefewfwefwefwe}

{In this section, we construct lax symmetric monoidal refinements of functors from symmetric monoidal $1$-categories to the categories $\Cat_{1}$  (and $\Add_{1}$) of small (additive) categories. }

\subsection{From $2$- to $\infty$-categories}

 A symmetric monoidal structure on a $1$-category $\bC$ consists of the tensor functor $$\otimes_{\bC}:\bC\times \bC\to \bC\ ,$$ the tensor unit $1_{\bC}$, and the associator, symmetry and unit-transformations, which must satisfy various compatibility relations.  
 If $\bC$ and $\bD$  are   symmetric monoidal  $1$-categories, then we can consider lax symmetric monoidal functors from $\bC$ to $\bD$. 
Such a lax symmetric monoidal  functor is given by  a functor  $F:\bC\to \bD$     together with
a natural transformation $$F(C)\otimes_{\bD} F(C')\to F(C\otimes_{\bC} C^{\prime})\ , \quad C,C^{\prime}\in \bC$$
{that} is compatible with the associators, symmetries and unit-transformations of $\bC$ and $\bD$ in a suitable way. We will list these structures and relations in Subsection \ref{ewfoiewfwefwefwef} below.

The categories $\Add$ or $\Cat$ of small additive categories and small categories are naturally $2$-categories.  Furthermore, the category $\Cat$ is symmetric monoidal with respect to the Cartesian symmetric monoidal structure $\times:=\otimes_{\Cat}$. The category $\Add$ has also a symmetric monoidal structure  $\otimes_{\Add}$: if $\bA$ and $\bB$ are two additive categories, then the objects of the tensor product $\bA\otimes_{\Add} \bB$
are pairs $(A,B)$ of objects   $A$ in $\bA$ and $B$ in $\bB$, and the morphisms are given by the tensor product
$$\Hom_{\bA\otimes_{\Add} \bB}((A,B),(A^{\prime},B^{\prime})):= \Hom_{\bA}(A,A^{\prime})\otimes_{\Z} \Hom_{\bB}(B,B^{\prime})$$ of abelian groups.

In the case of a  symmetric monoidal structure on a $2$-category, like $\Cat$ or $\Add$, {we have the same {compatibility} relations} between the structures {(tensor functor, tensor unit,  etc.)} {as} in the $1$-categorical case, {but they} are satisfied up to $2$-morphisms {only,} which in turn must satisfy higher compatibility relations. A similar remark applies to the notion of a (lax) symmetric monoidal functor.

In the present paper we consider the $1$-categorical situation as explicitly manageable, and we will avoid to explicitly work with symmetric monoidal structures on $2$-categories.  

Let $\bC$ be a symmetric monoidal $1$-category. Our goal is to construct symmetric monoidal functors
$F:\bC\to \Cat$ or $F:\bC\to \Add$ using $1$-categorical data only. Instead of working with the symmetric monoidal $2$-categories $\Cat$ or $\Add$ we will actually use the associated symmetric monoidal
$\infty$-categories $\Cat_{\infty}$ or  $\Add_{\infty}$. 

We  start with the
 ordinary category $\Cat_{1}$ of small categories.  Let $W_{\Cat}$ be the equivalences in $ \Cat_{1} $. The localization in large $\infty$-categories
$$\Cat_{\infty}:= \Nerve(\Cat_{1}) [W_{\Cat}^{-1}]$$
is the large $\infty$-category of categories. It models the $2$-category $\Cat$ in the following sense. 
The $2$-category $\Cat$  can be considered as a category enriched in categories.
Applying the nerve functor {$\Nerve$} to the $\Hom$-categories in $\Cat$
we get a fibrant\footnote{i.e.,  the $\Hom$-complexes are Kan complexes
} simplicially enriched category $\Nerve(\Cat)$. Applying the {homotopy} coherent nerve functor {$\mathcal{N}$,} we get an $\infty$-category 
$$\Nerve_{2}(\Cat):=\mathcal{N}(\Nerve(\Cat))\ . $$
Then, we have an equivalence of $\infty$-categories
$$\Nerve_{2}(\Cat)\simeq \Cat_{\infty}\ . $$
We refer to the appendix of \cite{GHN17} for more details about $\Nerve_{2}$.

The category $\Cat_{1}$ is a symmetric monoidal category and therefore gives rise to an op-fibration   of $1$-categories \cite[Constr. 2.0.01]{HA}, and {to} a symmetric monoidal $\infty$-category  \cite[Def. 2.0.0.7 \& Ex. 2.1.2.21]{HA}
$$\Cat_{1}^{\otimes}\to \Fin_{*}\ ,\text{ and }\quad \Nerve(\Cat_{1}^{\otimes})\to \Nerve(\Fin_{*}) \ ,$$  respectively. 
The equivalences $W_{\Cat}$ are preserved by the cartesian product. Hence we can form a symmetric monoidal localization \cite[Prop. 3.2.2]{hinich}     $$\Cat_{\infty}^{\otimes} :=\Nerve(\Cat_{1}^{\otimes})[W_{\Cat}^{\otimes,-1}] \to \Nerve(\Fin_{*})$$
 whose underlying $\infty$-category is equivalent to $\Cat_{\infty}$. Conseqently,   the symmetric monoidal $\infty$-category $\Cat_{\infty}^{\otimes}\to  \Nerve(\Fin_{*})$
 models the symmetric monoidal $2$-category $\Cat$. In this way we avoid to spell out the structures of a symmetric monoidal $2$-category explicitly.
 
A similar reasoning applies to $\Add$. We consider the {large} $1$-category $\Add_{1}$ of {small} additive categories and exact functors with the equivalences $W_{\Add}$. Then we define  the {large} $\infty$-category
$$\Add_{\infty}:=\Nerve(\Add_{1})[W_{\Add}^{-1}]$$ and get an equivalence   
$$\Add_{\infty}\simeq \Nerve_{2}(\Add)\ .$$ 
We can consider $\Add_{1}$ as a symmetric monoidal category giving rise to {an}  op-fibration of $1$-categories and  a symmetric monoidal $\infty$-category 
$$\Add_{1}^{\otimes} \to \Fin_{*}\ , \quad \Nerve(\Add_{1}^{\otimes}) \to \Nerve(\Fin_{*})\ .$$
Since the equivalences $W_{\Add}$  are   preserved by the tensor  product {$\otimes_{\Add}$,} we get the 
symmetric monoidal localization 
 $$\Add_{\infty}^{\otimes}:=\Nerve(\Add_{1}^{\otimes})[W_{\Add}^{\otimes,-1}] \to \Nerve(\Fin_{*})$$
 whose underlying $\infty$-category is equivalent to $\Add_{\infty}$. Therefore  $\Add_{\infty}^{\otimes}\to \Nerve(\Fin_{*})$
  models the symmetric monoidal $2$-category $\Add$.

Let $\bC$ be an ordinary category.
A functor $F:\bC\to \Cat_{1}$ (or $F:\bC\to \Add_{1}$) gives rise to a functor between $\infty$-categories
$F_{\infty}:\Nerve(\bC)\to  \Cat_{\infty}$ (or $F_{\infty}:\Nerve(\bC)\to \Add_{\infty}$) in the natural way, e.g. as the composition
$$F_{\infty}: \Nerve(\bC)\stackrel{\Nerve(F)}{\to} \Nerve(\Cat_{1})\to    \Nerve(\Cat_{1})[W_{\Cat}^{-1}]=\Cat_{\infty}\ .$$

A symmetric monoidal $1$-category $\bC$ gives rise to  the symmetric monoidal $\infty$-category  $\Nerve(\bC^{\otimes})\to \Nerve(\Fin_{*})$ whose underlying $\infty$-category is equivalent to $\Nerve(\bC)$. We now consider a functor 
 $F:\bC\to \Cat_{1}$ (or $F: \bC\to \Add_{1}$). {Recall that a map of $\infty$-operads \cite[Def. 2.1.2.7]{HA} can be thought of as a (lax) symmetric monoidal functor \cite[Def. 2.1.3.7]{HA} between the underlying categories.}

Let $F$ and $F_{\infty}$ be as above.

 \begin{ddd}\label{gfuzfkuzfkuzf}
 	A lax symmetric monoidal refinement of  $ F$ 
	is   a morphism of $\infty$-operads 
$$F^{\otimes}:\Nerve(\bC^{\otimes})\to \Cat^{\otimes}_{\infty}\ , \quad (F^{\otimes}:\Nerve(\bC^{\otimes})\to \Add^{\otimes}_{\infty})$$
{that} induces a functor equivalent to $F_{\infty}$ on the  underlying $\infty$-categories.   
\end{ddd}

Using this definition we avoid to spell out the details of the notion of a lax-symmetric functor from $\bC$ to the $2$-category $\Cat$ or $\Add$.

\subsection{Symmetric monoidal refinements of functors to $\Cat_{1}$ and $\Add_{1}$}

{In this subsection we state the technical results Theorem \ref{gui34tefrgeg} and Theorem \ref{rgio34t34r34rerg} which provide lax symmetric monoidal refinements of functors to $\Cat_{1}$ and $\Add_{1}$.}

Let $\bC$ be a $1$-category. A functor between  $1$-categories
$$F:\bC\to \Cat_{1}$$ can be interpreted, via the Grothendieck construction,  
as a cocartesian fibration
$$\pi_{F}:\cF\to \bC\ .$$ An object of the $1$-category 
$\cF$ is a pair $(X,A)$ with $X$ in $\bC$ and $A$ in $F(X)$. A morphism $(X,A)\to (Y,B)$ is a pair $(f,\phi)$ of a morphism
$f:X\to Y$ in $\bC$ and a morphim $\phi:F(f)(A)\to B$ in $F(Y)$. 

Assume that the categories $\bC$ and $\cF$ have symmetric monoidal structures such that \begin{equation}\label{werfqelknlweewwefw} 
 \pi_{F}((X,A)\otimes_{\cF}(X^{\prime},A^{\prime}))=X\otimes_{\bC}X^{\prime}\ ,  \end{equation} i.e., $\pi_{F}$ preserves the tensor product strictly. 
Then, we can write $$(X,A)\otimes_{\cF}(X^{\prime},A^{\prime})=(X\otimes_{\bC} X^{\prime}, A\boxtimes_{X,X^{\prime}}A^{\prime})\ .$$
 For every  two objects $X,X^{\prime}$ in $\bC$ we obtain a bifunctor 
\begin{equation}\label{f3foi34jfio34jf34f34f3f}
\boxtimes_{X,X^{\prime}}:F(X)\times F(X^{\prime})\to F(X\otimes_{\bC}X^{\prime})
\end{equation}
which is defined on morphisms in the canonical way.
Let $$f:X\to X^{\prime}\ , \quad g:Y \to Y^{\prime}$$ be morphisms in $\bC$ and $A$ in $F(X)$ and $B$ in $F(Y)$. Then
$$(f,\id_{F(f)(A)}):(X,A)\to (X^{\prime} ,F(f)(A))\ , \quad (g,\id_{F(g)(B)}):(Y,B)\to (Y^{\prime}, F(g)(B))$$  are  morphisms in $\cF$.
Then the second component of their tensor product 
$$(f,\id_{F(f)(A)})\otimes_{\cF} (g,\id_{F(g)(B)}) :(X\otimes_{\bC} Y,A\boxtimes_{X,Y}B)\to
(X^{\prime}\otimes_{\bC} Y^{\prime}, F(f)(A) \boxtimes_{X^{\prime},Y^{\prime}}F(g)(B )) $$
 is a morphism \begin{equation}\label{vwrkjbwwejkvjbnkwevewvw}
F(f\otimes_{\bC}g)(A\boxtimes_{X,Y}B)\to  F(f)(A) \boxtimes_{X^{\prime},Y^{\prime}}F(g)(B )
  \end{equation}
  in $F(X^{\prime}\otimes_{\bC}Y^{\prime})$.
  This morphism will appear in the assumptions of the two theorems below.

We now consider the following data:
\begin{enumerate}
\item {a symmetric monoidal $1$-category $\bC$,}
\item {a} functor $F:\bC\to \Cat_{1}$,
\item {a} symmetric monoidal structure   on the Grothendieck construction $\cF$ of $F$.
\end{enumerate}
Let $\pi_{F}:\cF\to \bC$ denote the associated projection.

\begin{theorem}\label{gui34tefrgeg}
 Assume:
 \begin{enumerate}
\item \label{regoijrogergergf2f} The functor $\pi_{F}$ {strictly} preserves   the tensor product,   the tensor unit as well as the associator, unit, and symmetry transformations.
\item\label{f2f3f23e23e}  \label{regoijrogergergf2f2} For every two  objects $(X,A)$ and $(Y,B)$ in $\cF$ and morphisms $f:X\to X^{\prime}$ and $g:Y\to Y^{\prime}$ in $\bC$ the   morphism 
\eqref{vwrkjbwwejkvjbnkwevewvw}
 $$F(f\otimes_{\bC}g)(A\boxtimes_{X,Y} B) \to  F(f)(A)\boxtimes_{X^{\prime}{,}Y^{\prime}} F(g)(B)\ .$$
is an isomorphism.
\end{enumerate}
Then the data provide   a lax symmetric monoidal refinement  (Def. \ref{gfuzfkuzfkuzf})
$$F^{\otimes}:\Nerve(\bC^{\otimes})\to \Cat^{\otimes}_{\infty}$$
of the functor $F$.
 \end{theorem}
Note that Condition \ref{regoijrogergergf2f} in the theorem implies the Relation \eqref{werfqelknlweewwefw} so that the bifunctors $\boxtimes_{X,Y}$
appearing in Condition \ref{regoijrogergergf2f2} are, in fact, defined.

The {analoguous} version for additive categories is the following.

Consider the following data:
\begin{enumerate}
\item {a symmetric monoidal $1$-category $\bC$,}
\item {a} functor $F:\bC\to \Add_{1}$,
\item {a} symmetric monoidal structure   on the Grothendieck construction $\cF$ of $F$.
\end{enumerate}
Let $\pi_{F}:\cF\to \bC$ denote the {associated} projection.

\begin{theorem}\label{rgio34t34r34rerg}
 Assume:
 \begin{enumerate}
\item The functor $\pi_{F}$  {strictly} preserves   the tensor product,   the tensor unit as well as the associator, unit, and symmetry transformations.
\item\label{gioreg34teergegegegeg} The functors $\boxtimes_{X,X^{\prime}}$ are bi-additive for every $X,X^{\prime}$ in $\bC$.
\item  For every two objects $(X,A)$ and $(Y,B)$ in $\cF$ and morphisms $f:X\to X^{\prime}$ and $g:Y\to Y^{\prime}$ in $\bC$ the   morphism 
\eqref{vwrkjbwwejkvjbnkwevewvw}
 $$F(f\otimes_{\bC}g)(A\boxtimes_{X,Y} B) \to  F(f)(A)\boxtimes_{X^{\prime},Y^{\prime}} F(g)(B)\ .$$
is an isomorphism.
\end{enumerate}
Then the data provide a lax symmetric monoidal refinement  
$$F^{\otimes}:\Nerve(\bC^{\otimes})\to \Add^{\otimes}_{\infty}$$
of the functor $F$.

\end{theorem}

\subsection{Proofs of Theorem \ref{gui34tefrgeg} and Theorem \ref{rgio34t34r34rerg}.}

We start with the proof of Theorem \ref{gui34tefrgeg}. 
Let  $$\Nerve(\bC^{\otimes})\to {\Nerve}(\Fin_{*})$$ denote the symmetric monoidal $\infty$-category  corresponding to 
the symmetric monoidal category $\bC$ \cite[Ex. 2.1.2.21]{HA}. Let $$\Cat^{\otimes}_{\infty}\to   {\Nerve}(\Fin_{*})$$ be the     cocartesian fibration corresponding to 
the symmetric monoidal category of small categories.
Then the $\infty$-category $$\Alg_{\Nerve(\bC)}(\Cat_{\infty}):=\mbox{\{operad maps $\Nerve(\bC^{\otimes})\to \Cat_{\infty}^{\otimes}$\}}$$ corresponds to the $\infty$-category of lax symmetric monoidal functors
$$\Nerve(\bC)\to \Cat_{\infty}\ ,$$ 
see the text after \cite[Rem. 2.1.3.6]{HA}.
We let $$\Mon_{\Nerve(\bC)}(\Cat_{\infty})$$ denote the category of $\Nerve(\bC)$-monoids in $\Cat_{\infty}$  
 \cite[Def. 2.4.2.1]{HA}. By  \cite[Prop. 2.4.2.5]{HA}, we have an equivalence
 $$\Alg_{\Nerve(\bC)}(\Cat_{\infty})\simeq \Mon_{\Nerve(\bC)}(\Cat_{\infty})\ .$$
By  \cite[Rem. 2.4.2.4]{HA}, in order to provide an object of  $\Mon_{\Nerve(\bC)}(\Cat_{\infty})$, it suffices to
present a cocartesian fibration $$p:\cC^{\otimes}\to \bC^{\otimes}$$ which exhibits
$\cC^{\otimes}$ as a   $\Nerve(\bC)$-monoidal category \cite[Rem. 2.1.2.13]{HA}.
To this end we must show that the composition
$$\cC^{\otimes}\to \bC^{\otimes}\to {\Nerve}(\Fin_{*})$$
exhibits $\cC^{\otimes}$ as an $\infty$-operad \cite[Prop. 2.1.2.12]{HA}.

Let $$\pi_{F}:\cF\to \bC$$ be a symmetric monoidal functor between $1$-categories as in the Theorem  \ref{gui34tefrgeg}. 
We get 
an induced functor of symmetric monoidal categories $$\pi_{F}^{\otimes}:\cF^{\otimes}\to \bC^{\otimes}$$ and {thus} a morphism of $\infty$-operads
 $$\Nerve(\pi^{\otimes}_{F}):\Nerve(\cF^{\otimes})\to \Nerve(\bC^{\otimes})\ .$$ Our task is then to show that 
 $\Nerve(\pi_{F}^{\otimes})$ exhibits 
$\Nerve(\cF^{\otimes})$ as an $\Nerve(\bC)$-monoidal category, and we must only check that
$\Nerve(\pi_{F}^{\otimes})$ is a cocartesian fibration. 
It suffices to check that $\pi_{F}^{\otimes}$ is an op-fibration of $1$-categories.
 
By assumption,  the underlying functor of   $\pi_{F} $ (after forgetting the symmetric monoidal structures)   arose from a Grothendieck construction for a functor $$F:\bC\to \Cat_{1}\ .$$  
Recall from \cite[Constr. 2.0.0.1]{HA} that the objects of $\bC^{\otimes}$ in the fibre  $\bC\langle n\rangle $ of $\bC^{\otimes}$ over $\langle n\rangle $ in $\Fin_{*}$ are $n$-tuples of objects of $\bC$.
Consider  two objects $$(X_{1},\dots,X_{n})\ , \quad (Y_{1},\dots,Y_{m})$$ in $\bC^{\otimes}\langle n\rangle $ and $\bC^{\otimes}\langle m\rangle $ and an object $$((X_{1},A_{1}),\dots,(X_{n},A_{n}))$$ in
$\cF^{\otimes}\langle n\rangle $, where $A_{i}$ belongs to $F(X_{i})$. Let $\alpha\langle n\rangle \to \langle m\rangle $ be a morphism in $\Fin_{*}$ and $$f: 
(X_{1},\dots,X_{n})\to (Y_{1},\dots,Y_{m})$$  be a morphism in $\bC^{\otimes}$ over $\alpha$. Then $f$ is given by   a collection of morphisms
$ f:=(f_{j})_{j\in \langle m\rangle }$  with $$f_{j}:\otimes_{i\in \alpha^{-1}(j)} X_{i}\to Y_{j}\ .$$  We must provide a cocartesian lift of $f$.  For $j$ in $\langle m\rangle$
we have a morphism 
$$g_{j}{:=}(f_{j},\id_{\boxtimes_{i\in \alpha^{-1}(j)}A_{i}}):(\otimes_{i\in \alpha^{-1}(j)} X_{i},\boxtimes_{i\in \alpha^{-1}(j)}A_{i})\to (Y_{j},F(f_{j})(\boxtimes_{i\in \alpha^{-1}(j)}A_{i}))$$ in $\cF$.
 One now checks in a straightforward (but tedious) manner that the collection
$g:=(g_{j})_{j\in \langle m\rangle }$ is the cocartesian lift of $f$. 
The argument repeatedly uses the Condition \ref{f2f3f23e23e}. This finishes the proof of Theorem \ref{gui34tefrgeg}.

 We now turn to the proof of Theorem \ref{rgio34t34r34rerg}. 
Consider the symmetric monoidal subcategory 
$$\Add_{\infty}^{\otimes}\subseteq \Cat_{\infty}^{\otimes}\ .$$

Indeed we can {first} consider the subcategory $\Cat_{\infty}(\coprod)$ of $\infty$-categories {which} admit finite coproducts and coproduct preserving functors.
By  \cite[Cor. 4.8.1.4]{HA} (applied to  the collection  {$\cK$} of finite sets) we get a symmetric monoidal subcategory
$$\Cat_{\infty}(\coprod)^{\otimes}\to  \Cat_{\infty}^{\otimes}\ .$$
In the next step we view
$\Add_{\infty}$ as a full subcategory of $\Cat_{\infty}(\coprod) $ of pointed $1$-categories in which products and coproducts coincide. Using  \cite[Cor. 2.2.1.1]{HA} one then shows that
$$\Add_{\infty}^{\otimes}\to  \Cat_{\infty}(\coprod)^{\otimes}$$
is again a suboperad.
We now consider the diagram
$$\xymatrix{\Add_{\infty}^{\otimes}\ar[r]& \Cat_{\infty}(\coprod)^{\otimes}\ar[d]\\\Nerve(\bC^{\otimes})\ar[r]\ar@{-->}[u]\ar@{..>}[ur]&\Cat_{\infty}^{\otimes}}\ .$$
The lower horizontal map is a morphism of {$\infty$-}operads by Theorem \ref{gui34tefrgeg}. 
We first argue that the dotted lift exists. To this end we use  \cite[Notation 4.8.1.2]{HA}. One must check that
 $F$ takes values in categories admitting finite coproducts (clear), and that the functors
$$\boxtimes_{X,Y}:F(X)\times F(Y)\to F(X\times Y)$$ preserves 
sums in both variables separately, i.e., Assumption \ref{gioreg34teergegegegeg}.
Finally,  for the dashed arrow we use that 
$F$ takes values in $\Add_{{1}}$.

\section{The symmetric monoidal functor of controlled objects}

\subsection{Symmetric monoidal structures}\label{ewfoiewfwefwefwef}

In this subsection we write out, for later reference, the structures of a symmetric monoidal category and of a (lax) symmetric monoidal functor.
Let $\bC$ be a $1$-category: \begin{ddd}\label{gui23r32reger} {\cite[Sec. VII. 1. \& 7.]{maclane}}
A symmetric monoidal structure on $\bC$ is given by the following data:
\begin{enumerate}
\item\label{gui23r32reger1} a bifunctor $(-\otimes_{\bC}-):\bC\times \bC\to \bC$,
\item\label{gui23r32reger2} an object $1_{\bC}$ (the tensor unit),
\item a natural isomorphism  (the associativity constraint) $$\alpha^{\bC}:(-\otimes_{\bC}-)\circ ((-\otimes_{\bC}-) \times  \id_{\bC} )  \to (-\otimes_{\bC}-)\circ  (\id_{\bC}\times (-\otimes_{\bC}-))   \ ,$$
\item a  natural isomorphism $\eta^{\bC}:1_{\bC}\otimes_{\bC}-\to \id_{\bC}$ (the unit constraint),
\item a  natural isomorphism (the symmetry) $\sigma^{\bC}:(-\otimes_{\bC}-)\circ T\to (-\otimes_{\bC}-)$, where $T:\bC\times \bC\to \bC\times \bC$ is the flip functor.
\end{enumerate}
This data have to satisfy the following relations:
\begin{enumerate}
\item the pentagon relation,
\item  the triangle relation,
\item the inverse relation, 
\item the associativity coherence. 
\end{enumerate}
\end{ddd}
A symmetric monoidal category is a category {equipped} with a symmetric monoidal structure. 

We will use the name of the category as a superscript for the constraints, but if we evaluate e.g. the symmetry constraint $\sigma^{\bC}$ at the objects $C,C^{\prime}$ of $\bC$, then we write shortly $\sigma_{C,C^{\prime}}$ instead of $\sigma^{\bC}_{C,C^{\prime}}$ since the type of objects in the subscript already determines the category in question.

Let $\bC$ and $\bD$ be   symmetric monoidal categories,  and let $F:\bC\to \bD$ be a functor.
\begin{ddd}\label{34f8924fregerg} {\cite[Sec. XI. 2.]{maclane}}
A symmetric monoidal structure on $F$ is given by the following data:
\begin{enumerate}
\item\label{3gio34t3434g} an isomorphism  $\epsilon^{F}:1_{\bD}\to F(1_{\bC})$,
\item\label{3gio34t3434g1} a natural isomorphism $\mu^{F}:(- \otimes_{\bD}-)\circ (F\times F)\to F\circ ( -\otimes_{\bC}-)$.
\end{enumerate}
This data have to satisfy the following relations:
\begin{enumerate}
\item associativity relation,

\item\label{fiowfewef23rr3} unitality relation, 
\item\label{evfiwehfioewfewfewfewff} symmetry relation. 
\end{enumerate}

\end{ddd}

\begin{rem}\label{gir2i23joijoir23r23r}
If we weaken the {assumptions} and {we} only require   that $\epsilon^{F}$ and $\mu^{F}$ are
natural transformations, then we get the definition of a lax symmetric monoidal functor. 
\end{rem}

%
%

\subsection{Bornological coarse spaces }\label{geriogjheroigfwefwefw}

In this subsection we recall the definition {of} the symmetric monoidal category $G\BC$ of $G$-bornological coarse spaces \cite[Sec. 2]{buen}, \cite[Sec. 2.1]{equicoarse}.

In the  definitions below we will use the following {notation:}
\begin{enumerate}
\item For a set $Z$ we let $\cP(Z)$ denote the power set of $Z$. 
\item If  a group $G$ acts on a set $X$, then it acts diagonally on $X\times X$ and therefore on $\cP(X\times X)$.
For $U$ in $\cP(X\times X)$ we set \begin{equation*}\label{betrboklvree}
GU:=\bigcup_{g\in G} gU\ .
\end{equation*} 
\item For $U$ in $\cP(X\times X)$ and $B$ in $\cP(X)$ we define the
 $U$-thickening $U[B]$   by
\begin{equation}\label{viojoewfwefvvdvsdv} U[B]:=\{x\in X \mid \exists y\in B: (x,y) \in U \}\ .\end{equation}
\item For $U$ in $\cP(X\times X)$ we define {its} inverse  by $$U^{-1}:=\{(y,x) \mid (x,y)\in U  \}\ .$$  \item For $U,V$ in $\cP(X\times X)$ we define their composition  by \begin{equation}\label{rv3roih43iuoff3fwe}
U\circ V:=\{(x,z) \mid \exists y\in X : (x,y)\in U \wedge (y,z)\in V  \}\ .
\end{equation}
 
\end{enumerate}

Let $G$ be a group and let $X$ be a $G$-set.

\begin{ddd}
 A {$G$-coarse structure} $\cC$ on $X$ is a subset of $\cP(X\times X)$ with the following properties:
\begin{enumerate}
\item $\cC$ is closed under composition, inversion, and forming finite unions or subsets.
\item $\cC$ contains the diagonal $\diag(X)$ of $X$.
\item {For} every $U$ in $\cC$, the set $GU$
is also in $\cC$. 
\end{enumerate}
The pair $(X,\cC)$ is called a  {$G$-coarse space},
 {and the members of $\cC$ are called (coarse)  {entourages} of $X$.}
\end{ddd}

Let $(X,\cC)$ and $(X^{\prime},\cC^{\prime})$ be $G$-coarse spaces and  {let} $f\colon X\to X^{\prime}$ be an equivariant map between the underlying sets.
\begin{ddd}
The map $f$ is  {controlled} if for every $U$ in $\cC$ we have  $(f\times f)(U)\in \cC^{\prime}$.  \end{ddd}

We obtain a category $G\Coarse$ of $G$-coarse spaces and controlled equivariant maps.

Let $G$ be a group and let $X$ be a $G$-set.
\begin{ddd}\label{erjgoijgoiergjergegrrr}A  {$G$-bornology} $\cB$ on  $X$ is a subset of $\cP(X)$ with the following properties:
\begin{enumerate}
\item $\cB$ is closed under {forming} finite unions and subsets.
\item $\cB$ contains all finite subsets of $X$.
\item $\cB$ is $G$-invariant. 
\end{enumerate} 
The pair $(X,\cB)$ is called a  {$G$-bornological  space}{, and the members of $\cB$ are called  {bounded subsets} of $X$.}
\end{ddd}

Let $(X, {\cB})$ and $(X^{\prime}, {\cB^{\prime}})$ be $G$-bornological spaces and {let} $f \colon X\to X^{\prime}$ be an equivariant map between the underlying sets.

\begin{ddd}\label{erjgoijgoiergjergegrrr1}
 {The map} $f$ is {proper} if for every $B^{\prime}$ in $\cB^{\prime}$ we have $f^{-1}(B^{\prime})\in  \cB$.   \end{ddd}
 
We obtain a category $G\Born$ of $G$-bornological spaces and proper equivariant maps.

Let $X$ be a $G$-set {equipped} with a $G$-coarse structure $\cC$ and a $G$-bornology $\cB$.
\begin{ddd}
The coarse structure $\cC$ and the bornology $\cB$ are said to be  
  {compatible} if 
for every $B$ in $\cB$ and $U$ in $\cC$ 
 {the $U$-thickening $U[B]$ (see \eqref{viojoewfwefvvdvsdv}) lies in $\cB$.}
\end{ddd}

\begin{ddd}\label{ergergregreg5t45t}
A  {$G$-bornological coarse space} is a triple $(X,\cC,\cB)$ consisting of a $G$-set $X$, a $G$-coarse structure $\cC$ and a $G$-bornology $\cB$ {on $X$,} such that $\cC$ and $\cB$  are  compatible. \end{ddd}

Usually we will denote a $G$-bornological coarse space by the symbol $X$ and write $\cB(X)$ and $\cC(X)$ for its bornology and coarse structures.


\begin{ddd}
A  {morphism} $f:X \to X^{\prime} $ between $G$-bornological coarse spaces
is an equivariant  map  of the {underlying} $G$-sets {that} is controlled and proper.
\end{ddd}

We obtain a category $G\BC$ of $G$-bornological coarse spaces and morphisms.

Next we describe the symmetric monoidal structure on $G\BC$ \cite[Ex. 2.17]{equicoarse}. We have a forgetful functor
$$U:G\BC\to G\Set$$ which associates to every $G$-bornological coarse space $X$ {its} underlying $G$-set. 
This functor is faithful. 
The category $G\Set$ {is endowed with} the cartesian symmetric monoidal structure.  The symmetric monoidal structure on $G\BC$ will be defined {in such a way}  that the functor $U$ preserves the unit and the tensor product strictly, i.e., the morphisms  \ref{3gio34t3434g} and \ref{3gio34t3434g1} in Definition \ref{34f8924fregerg} are identities. In other words, the associator, unit and symmetry constraints are imported from $G\Set$ and 
satisfy the relations required in Definition \ref{gui23r32reger} automatically. 

We start with the description of the bifunctor
$$-\otimes_{G\BC}-:G\BC\times G\BC\to G\BC\ .$$
Let $X $ and  $ X^{\prime} $ be two $G$-bornological coarse spaces. Then
their tensor product $$ X\otimes_{G\BC}  X^{\prime} $$  is the $G$-bornological coarse spaces defined as follows:
\begin{enumerate}
\item The underlying $G$-set of $X\otimes_{G\BC}  X^{\prime}$ is the cartesian product of the underlying $G$-sets $X\times X^{\prime}$.  
\item The $G$-bornology on $X\times X^{\prime}$ is generated by the subsets $B\times B^{\prime}$ for all $B$ in $\cB(X)$ and $B^{\prime}$ in $\cB(X^{\prime})$.
\item The $G$-coarse structure on $X\times X^{\prime}$ is generated by the entourages $U\times U^{\prime}$ for $U$ in $\cC(X)$ and $U^{\prime}$ in $\cC(X^{\prime})$.
\end{enumerate}
Here a $G$-bornological (or coarse, respectivley) structure generated by a family of   subsets (or entourages)  is the minimal $G$-bornological (or $G$-coarse) structure containing these subsets (or entourages). 
Note that the underlying $G$-coarse space of the tensor product  represents the cartesian product of the underlying $G$-coarse spaces of the factors in $G\Coarse$, but the tensor product is not the cartesian product in $G\BC$ in general.  

From now on we will use the shorter notation $X\otimes X^{\prime}$ for the tensor product {of $G$-bornological coarse spaces}, i.e., we omit the subscript {$G\BC$.} 

If $f:X\to Y$ and $f^{\prime}:X^{\prime}\to Y^{\prime}$ are morphisms of $G$-bornological coarse spaces, then their tensor product
$$f\otimes f^{\prime}:X\otimes Y\to X^{\prime}\otimes Y^{\prime}$$ is induced by {the}  equivariant map of underlying $G$-sets $(x,y)\mapsto (f(x),f(y))$.
This finishes the description of the bifunctor \ref{gui23r32reger}.\ref{gui23r32reger1}

The tensor unit $1_{G\BC}$ (\ref{gui23r32reger}.\ref{gui23r32reger2})   is given  by the one-point space $*$. 

As explained above, the associativity, unit and symmetry constraints are imported from $G\Set$. It is straightforward to check that they are implemented by  morphisms of $G$-bornological coarse spaces.

This finishes the description of the symmetric monoidal structure $\otimes$ on the category  $G\BC$.

\subsection{Controlled objects}

In this section, for every additive category $\bA$ with a strict $G$-action, we describe the functor $$\bV_{\bA}^{G}:G\BC\to \Add_{1}$$ which sends a $G$-bornological coarse space $X$ to its additive category $\bV_{\bA}^{G}{(X)}$ of equivariant $X$-controlled $\bA$-object \cite[Sec.~8.2]{equicoarse}.
 
{For a group $G$, l}et $BG$ be the category with one object $*$ and $\End_{BG}(*)\cong G$. Then  $\Fun(BG,\Add_{1})$ is the category of additive categories with {a} strict $G$-action. Explicitly, an additive category with a strict $G$-action is an additive category $\bA$ (the evaluation of the functor at the object $*$ in $BG$)  together with an action of $G$ on $\bA$ by exact functors, which is strictly associative. Our notation   for the action of $g$ in $G$ on objects $A$ {of $\bA$}  and morphisms $f$ is
$$(g,A)\mapsto gA\ , \quad (g,f:A\to A^{\prime})\mapsto (gf:gA\to gA^{\prime})\ .$$

  Let $\bA$ be an additive category with a strict $G$-action and $X$ be a $G$-bornological coarse space. We consider the bornology  $\cB(X)$ of $X$ as a poset with a $G$-action $(g,B)\mapsto gB$,  hence as a category with a strict $G$-action, i.e., an object of  $\Fun(BG,\Cat_{1})$.

The category $\Fun(\cB(X),\bA)$ has an induced $G$-action which can explicitly be described as follows. If $M:\cB\to \bA$ is a functor and $g$ is an element of $G$, then  $g M\colon \cB(X)\to \bA$ is  the functor which sends a bounded set $B$ in $\cB(X)$ to the object $gM(g^{-1}(B))$ of $\bA$.
If $\rho:M\to  M^{\prime}$ is a natural transformation between two such functors, then we let  $g\rho :gM\to gM^{\prime}$  denote the canonically induced natural transformation.

\begin{ddd}\label{def:Xcontrolledobject}{\cite[Def. 8.3]{equicoarse}}
	An {equivariant $X$-controlled $\bA$-object} is a pair $(M,\rho)$ consisting of a functor
	$M \colon \cB(X) \to \bA$
	and a family $\rho=(\rho(g))_{g\in G}$ of natural isomorphisms 
	$\rho(g)\colon M \to g M$
	 	satisfying the following conditions:
	\begin{enumerate}
		\item\label{def:Xcontrolledobject:it1} $M(\emptyset) \cong 0$.
		\item\label{def:Xcontrolledobject:it3} For all $B, B'$ in $\cB(X)$, the commutative square
		\[\xymatrix{
			M(B \cap B')\ar[r]\ar[d] & M(B)\ar[d] \\
			M(B')\ar[r] & M(B \cup B')
		}\]
		is a pushout square.
		\item\label{def:Xcontrolledobject:it4} For all $B$ in $\cB(X)$ there exists a finite subset $F$ of $B$ such that the inclusion $F\to B$ induces an isomorphism   $M(F) \xrightarrow{\cong} M(B)$. 		\item \label{def:Xcontrolledobject:it5}For all pairs of elements  $g,g'$ of $G$ we have the relation
		$\rho(g g^{\prime})= g\rho(g')\circ \rho(g)$.
		\qedhere
	\end{enumerate}
\end{ddd}

If $U$ is an invariant coarse entourage of $X$, i.e.,  an element of $\cC(X)^{G}$, then we get a $G$-equivariant functor
$$U[-]:\cB(X)\to \cB(X)$$ which sends a bounded
 subset $B$ of $X$ to its  $U$-thickening   $U[B]$, see \eqref{viojoewfwefvvdvsdv}. Indeed,  the $U$-thickening   $U[B]$ of a bounded subset $B$    is again bounded by the compatibility of the coarse structure $\cC(X)$ and the bornology $\cB(X)$ of $X$, and $U[-]$ preserves the inclusion relation. Since  $U$  is $G$-invariant we have   the equality  $U[gB]=gU[B]$. It implies that $U[-]$  is $G$-{equivariant}. 
 If $M:\cB(X)\to \bA$ is a functor, then we write $U[-]^{*}M:=M\circ U[-]$ for the pull-back of $M$ along $U[-]$.

Let $(M,\rho), (M',\rho')$ be two equivariant $X$-controlled $\bA$-objects and $U$ be  an invariant coarse entourage of $X$.
\begin{ddd}\label{rebgieorgergergegerg}
	An {equivariant $U$-controlled morphism} $\phi  \colon (M,\rho) \to (M',\rho')$ is a natural transformation
	\[ \phi \colon M  \to U[-]^{*}M' \ ,\] 
	such that $\rho'(g)\circ \phi =(g \phi )\circ \rho(g)$ for all elements $g$ of $G$.
\end{ddd}
We let $\Mor_U((M,\rho),(M',\rho'))$ denote the abelian group of equivariant $U$-controlled morphisms.

If $U^{\prime}$ is in $\cC(X)^{G}$ such that $U\subseteq U^{\prime}$, then for every $B$ in $\cB(X)$ we have $U[B]\subseteq U^{\prime}[B]$. These inclusions induce a transformation between functors
$U[-]^{*}M^{\prime} \to U^{\prime}[-]^{*} M^{\prime} $ and therefore a map 
$$\Mor_U((M,\rho),(M',\rho'))\to \Mor_{U^{\prime}}((M ,\rho),(M^{\prime},\rho'))$$
by postcomposition. Using these maps in the interpretation of the colimit
 we define the abelian group of equivariant controlled morphisms from $(M,\rho)$ to $(M',\rho^{\prime})$ by
\[ \Hom_{\bV^{G}_\bA(X)}((M,\rho),(M',\rho')) := \colim_{U \in \cC(X)^{G}} \Mor_U((M,\rho),(M',\rho'))\ .\]

We now consider a pair  of morphisms in $$ \Hom_{\bV^{G}_\bA(X)}((M,\rho),(M',\rho'))\:\:\mbox{and}\:\:  \Hom_{\bV^{G}_\bA(X)}((M^{\prime},\rho^{\prime}),(M^{\prime\prime},\rho^{\prime\prime}))\ ,$$ respectively, 
 which are represented by 
  $$\phi :M \to U[-]^{*}M^{\prime} \:\:\mbox{and}\:\: \phi ^{\prime}:M^{\prime} \to U^{\prime}[-]^{*}M^{\prime\prime} \ .$$ 
We define the composition of the two morphisms  to be represented by the morphism
$$U[-]^{*}\phi ^{\prime}\circ \phi : M\to (U^{\prime}\circ U)[-]^{*}M^{\prime\prime} \ , $$ (see \eqref{rv3roih43iuoff3fwe} for notation)
where
$$U[-]^{*}\phi^{\prime}:U[-]^{*}M^{\prime} \to (U^{\prime}\circ U)[-]^{*}M^{\prime\prime} $$ is defined in the canonical manner.
We denote the resulting   category of equivariant $X$-controlled $\bA$-objects and equivariant controlled morphisms by $\bV_\bA^G(X)$. This category is additive {\cite[Lemma 8.7]{equicoarse}}.


Let $f \colon X  \to X'  $  be a morphism of $G$-bornological coarse spaces,
and let $(M,\rho)$ be an equivariant $X$-controlled $\bA$-object.
Since $f$ is proper, it induces an equivariant functor $f^{-1}:\cB(X^{\prime})\to \cB(X)$, and we can define a functor $f_*M \colon \cB(X^{\prime}) \to \bA$ by
\[ f_*M := M\circ f^{-1} \ .\]
Furthermore, we define
\[f_*\rho(g) :=\rho(g)\circ f^{-1} \ .\]
Let $U$ be in $\cC(X)^{G}$ and let $\phi \colon (M,\rho) \to (M',\rho')$ be an equivariant $U$-controlled morphism. Then  $V:= (f \times f)(U)$ belongs to $\cC(X^{\prime})^{ G}$ and we have  $U[f^{-1}(B^{\prime})] \subseteq f^{-1}(V[B^{\prime}])$ for all bounded subsets $B^{\prime}$ of $X^{\prime}$. Therefore, we obtain an induced $V$-controlled morphism
\[ f_*\phi = \{  f_*M(B^{\prime})\stackrel{\phi_{f^{-1}(B^{\prime})}}{ \to} M(U[f^{-1}(B^{\prime})])\to  f_*M(V[B^{\prime}]) \}_{B^{\prime} \in \cB(X')}\ .\]
One checks that this construction defines an additive functor
\[ f_* \colon \bV^G_\bA(X) \to \bV^G_\bA(X')\ .\]
This completes the construction of the functor  \begin{equation}\label{ekrjekrferferferf}
\bV^G_\bA \colon G\BC \to \Add_{{1}}\ .
\end{equation}

In the following we give a more explicit description of the objects and morphisms in $\bV_{\bA}^{G}(X)$ which will be used in the description of the  symmetric monoidal structure on the Grothendieck construction {associated to the functor} $\bV_{\bA}^{G}$ in Section \ref{vevgebvveerverv}.

\begin{conv}\label{choosesum}
	We consider an additive category  $\bA$. If $(A_{i})_{i\in I}$ is a  family of objects of $\bA$ with at most finitely many non-zero members, then we use the symbol
	$\bigoplus_{i\in I} A_{i}$ in order to 
	denote a choice of an object of $\bA$  {together}  with a family of morphisms $(A_{j}\to  \bigoplus_{i\in I} A_{i})_{j\in I}$ {representing} the coproduct of the family.

Since in an additive category coproducts and products coincide, for every $j$ in $I$ we  furthermore have a   canonical projection   $$ \bigoplus_{i\in I} A_{i} \to A_{j}$$  
such that the diagram
 $$\xymatrix{A_{j}\ar[r]\ar@/^0.7cm/[rr]^{{\id_{A_{j}}}}&\bigoplus_{i\in I} A_{i} \ar[r]&A_{j}}$$
 commutes.

If $(A_{i^{\prime}}^{\prime})_{i^{\prime}\in I^{\prime}}$ is a second   family of this type 
and $(\phi_{i,i^{\prime} }:A_{i^{\prime}}^{\prime}\to A_{i})_{i^{\prime}\in I^{\prime},i\in I}$ is a family of morphisms in $\bA$, then we have a
unique  morphism $\oplus \phi_{i,i^{\prime}}$   such that the squares \begin{equation}\label{vejebvierhviervervev}
\xymatrix{A_{i^{\prime}}^{\prime}\ar[r]^{\phi_{i,i^{\prime}}}\ar[d]&A_{i}\ar[d]\\\bigoplus_{i^{\prime}\in I^{\prime}} A_{i^{\prime}}^{\prime} \ar[r]^{\oplus \phi_{i,i^{\prime}}}&\bigoplus_{i\in I} A_{i}}
\end{equation} 
commute for every $i^{\prime}$ in $I^{\prime}$ and $i$ in $I$.\hB 
\end{conv}
%

%

{Let $\bA$ be a small additive category with strict $G$-action.}
Let $X  $ be a  $G$-bornological coarse space  (see Definition \ref{ergergregreg5t45t}),  and let $(M,\rho)$ be an equivariant $X$-controlled $\bA$-object (see Definition \ref{def:Xcontrolledobject}). Let $B$ be  in $\cB(X)$ and $x$ be a point in $B$. The inclusion $\{x\}\to B$ induces a morphism
$M(\{x\})\to M(B)$ in $\bA$. The 
conditions \ref{def:Xcontrolledobject}.\ref{def:Xcontrolledobject:it1} and \ref{def:Xcontrolledobject}.\ref{def:Xcontrolledobject:it3}
together imply that
$M(\{x\})=0$ for all but finitely many points of $B$, and that the canonical morphism (induced by the universal property of the coproduct in $\bA$) 
\begin{equation}\label{g5giojgoigg5g45gfg}
\bigoplus_{x\in B} M(\{x\}) \xrightarrow{\cong} M(B)\ 
\end{equation}   
{is an isomorphism.}

Let now $U$ be in $\cC(X)^{G}$, and let $\phi:(M,\rho)\to (M^{\prime},\rho^{\prime})$ be an equivariant  $U$-controlled morphism. By Definition \ref{rebgieorgergergegerg}, {the morphism $\phi$} is given by   a natural transformation of functors $\phi:M\to U[-]^{*}M^{\prime}$ satisfying an equivariance condition.
For every point $x$ in $X$ we get a morphism
\begin{equation}\label{dewd23d2d3d2}
M(\{x\})\to M^{\prime}(U[\{x\}])\stackrel{\eqref{g5giojgoigg5g45gfg}}{\cong}
\bigoplus_{x^{\prime}\in U[\{x\}]} M(\{x^{\prime}\})\ 
\end{equation} 
{in $\bA$.}
We let\begin{equation}\label{vrtvoijogi43ferv}
\phi_{x^{\prime},x}:M(\{x\})\to M^{\prime}(\{x^{\prime}\})
\end{equation}   denote the composition of  {\eqref{dewd23d2d3d2}} with the projection onto the summand corresponding to $x^{\prime}$.
In this way we get a family of morphisms $(\phi_{x^{\prime},x})_{x^{\prime},x\in X}$ in $\bA$. In a similar manner, for $g$ in $G$, the transformation $\rho(g):M\to gM$ gives rise to a family of morphisms
\begin{equation}\label{v4rr4vhu34fr3f}
\big(\rho(g)_{x}:M(\{x\})\to gM(\{g^{-1}x\})\big)_{x\in X}\ .
\end{equation} 
By construction the family    $(\phi_{x^{\prime},x})_{x^{\prime},x\in X}$ satisfies the following conditions. 
\begin{enumerate}
	\item\label{efghewufw23} For all $x,x^{\prime}$ in $X$ the condition
	$\phi_{x^{\prime},x}\not=0$ implies that  $(x^{\prime},x )\in U$.  \item\label{efghewufw232} We have  
	$\rho^{\prime}(g)_{x^{\prime}}\circ \phi_{x^{\prime},x}=
	(g\phi)_{g^{-1}x^{\prime},g^{-1}x}\circ \rho(g)_{x}$ 
	for all $x,x^{\prime}$ in $X$ and $g$ in $G$.
\end{enumerate}

\begin{lem}\label{fiofhoiwefhwef32r} We have  a bijection between equivariant $U$-controlled morphisms $\phi:(M,\rho)\to (M^{\prime},\rho^{\prime})$
	and families
	$(\phi_{x^{\prime},x})_{x^{\prime},x\in X}$ of morphisms as in \eqref{vrtvoijogi43ferv} satisfying the {C}onditions  \ref{efghewufw23} and \ref{efghewufw232}. 
	
	%
	%
	%
\end{lem}

\begin{proof}
Let $(M,\rho)$  and $(M^{\prime},\rho^{\prime})$ be in $\bV_{\bA}^{G}(X)$.
We must show that a matrix $(\phi_{x^{\prime},x})_{x^{\prime},x\in X}$ of morphisms 
 as in \eqref{vrtvoijogi43ferv} which satisfies  the {C}onditions  \ref{efghewufw23} and \ref{efghewufw232}
 gives rise to a{n equivariant controlled} morphism $\phi:(M,\rho)\to (M^{\prime},\rho^{\prime})$. Let $U$ be in $\cC^{G}(X)$ such that {C}ondition  \ref{efghewufw23} holds true. We must construct an equivariant  natural transformation
 $\phi:M\to U[-]^{*}M^{\prime}$. 
 
 We consider
 $B$ in $\cB(X)$. Then
 $(M(\{x\})_{x\in B}$ and $(M^{\prime}(\{x^{\prime}\}))_{x^{\prime}\in U[B]}$ are families of objects in $\bA$ with at most finitely many non-zero members.
 Using Convention \ref{choosesum},  and in particular  the notation from \eqref{vejebvierhviervervev},
  we  can define the morphism
 $\phi_{B}:M(B)\to M^{\prime}(U[B])$ such that
 $$\xymatrix{\bigoplus_{x\in B}M(\{x\})\ar[rr]^{\oplus \phi_{x^{\prime},x}}\ar[d]^{\eqref{g5giojgoigg5g45gfg}}_{\cong}&&\bigoplus_{x^{\prime}\in U[B]}M^{\prime}(\{x^{\prime}\})\ar[d]^{\eqref{g5giojgoigg5g45gfg}}_{\cong}\\M(B)\ar[rr]^{\phi_{B}}&&M^{\prime}(U[B])}$$
 commutes. It is now straightforward to check that the family $(\phi_{B})_{B\in \cB(X)}$ assembles to a natural transformation $\phi:M\to U[-]^{*}M^{\prime}$ as required. By construction the morphism $\phi$ is $U$-controlled.  Furthermore, the Condition \ref{efghewufw232} implies that $\phi$ satisfies the equivariance condition stated in {Definition} \ref{rebgieorgergergegerg}. \end{proof}

Let ${f}:X_{0}\to X_{1}$ be a morphism of $G$-bornological coarse spaces and $(M_{i},\rho_{i})$ be  objects of $\bV_{\bA}^{G}(X_{i})$ for $i=0,1$.
Then a morphism \begin{equation}\label{f3oifjiocvcwecw}
\phi:f_{*}(M_{0},\rho_{0})\to (M_{1},\rho_{0})
\end{equation}
induces a  matrix \begin{equation}\label{f34flk3n4fok34f34f}
\left(\phi^{f}_{x_{1},x_{0}}:M_{0}(\{x_{0}\})\to M_{1}(\{x_{1}\})\right)_{x_{0}\in X_{0}, x_{1}\in X_{1}} \ .
\end{equation}
To this end we observe that
$$(f_{*}M_{0})(\{x^{\prime}_{1}\})=M_{0}(f^{-1}(\{x_{1}^{\prime}\}))\stackrel{\eqref{g5giojgoigg5g45gfg}}{\cong} \bigoplus_{x_{0}\in f^{-1}(\{x_{1}^{\prime}\})} M_{0}(\{x_{0}\})$$ so that
$$\left(\phi_{x_{1} ,x_{1}^{\prime}}:=\oplus_{x_{0}\in f^{-1}(\{x_{1}^{\prime}\})} \phi^{f}_{x_{0},x_{1}}:M_{0}(f^{-1}(\{x_{1}^{\prime}\}))\cong\bigoplus_{x_{0}\in f^{-1}(\{x_{1}^{\prime}\})} M_{0}(\{x_{0}\})
 \to M_{1}(\{x_{1}\})\right)_{x_{1}^{\prime},x_{1}\in X_{1}}$$
is the matrix representing $\phi$ according to Lemma \ref{fiofhoiwefhwef32r}.  As a consequence of  Lemma \ref{fiofhoiwefhwef32r} we obtain:
\begin{kor}\label{fuewihfiuwehfewfewfwef}A matrix \eqref{f34flk3n4fok34f34f}  represents a morphism \eqref{f3oifjiocvcwecw}
 iff the following conditions are satisfied: 
\begin{enumerate}
\item There exists an entourage $U_{1}$ in $\cC( X_{1})$ such that for every $x_{0}$ in $X_{0}$ and $x_{1}$ in $X_{1}$ the condition
$\phi^{f}_{x_{1},x_{0}}\not=0$ implies that $(x_{1},f(x_{0}) )\in U_{1}$.
\item For every $g$ in $G$ we have the equality
$$\rho_{1} (g)_{x_{1}}\circ \phi^{f}_{x_{1} ,x_{0}}=
	(g\phi^{f})_{g^{-1}x_{1},g^{-1}x_{0}}\circ \rho(g)_{x_{0}} 
\ .$$
\end{enumerate}
\end{kor}

\subsection{The symmetric monoidal  refinement of $\bV_{\bA}^{G}$}\label{vevgebvveerverv}

\newcommand{\cVA}{\mathcal{V}_{\bA}^{G}}

Let $\bA$ be a small  additive category with a strict $G$-action.
Then we let $$\pi:\cVA\to G\BC$$ denote the {Grothendieck} construction associated to the functor $\bV_{\bA}^{G}$ {(see \eqref{ekrjekrferferferf})}  viewed as a functor from $G\BC$ to $\Cat_{1}$. The goal of this section is the construction of a symmetric monoidal structure (see Definition \ref{gui23r32reger}) on $\cVA$ which satisfies the assumptions of Theorem \ref{rgio34t34r34rerg}.

\begin{ass}\label{giwrogerfwewf}
We assume that $\bA$ has a symmetric monoidal structure and that the strict  action of $G$ on $\bA$ has a refinement to an action by symmetric monoidal functors.  
\end{ass}

In order to introduce the notation for later arguments, we spell out the Assumption \ref{giwrogerfwewf} explicitly.
According to Definition \ref{gui23r32reger}   the category  $\bA$ comes with the following data:
\begin{enumerate}
\item a bifunctor $-\otimes_{\bA}-$, 
\item a tensor unit $1_{\bA}$,
\item an associativity constraint $\alpha^{\bA}$,
\item a unit constraint $\eta^{\bA}$,
\item a symmetry constraint $\sigma^{\bA}$.
\end{enumerate}
This data satisfy the {relations}   specified in Definition \ref{gui23r32reger}.
 
The strict action of $G$ on $\bA$ by symmetric monoidal functors   is implemented by the following data. For every $g$ in $G$ we have:
\begin{enumerate}
\item  an additive functor
$g:\bA\to \bA$,
\item\label{fweiufeiwofewfewf} an isomorphism $\epsilon^{g}:1_{\bA}\to g1_{\bA}$,
\item\label{vuhiwefwefwefewfwef2r} a natural isomorphism $\mu^{g}:(g-\otimes_{\bA}g-)\to g(-\otimes_{\bA}-)$,
\end{enumerate}
satisfying the {relations}  specified in Definition \ref{34f8924fregerg}.
We require that for all $g$ and $h $ in $G$   the following relation between the composition of symmetric monoidal functors and multiplication in $G$ holds true:
\begin{equation}\label{rgkjbkrjg34f34f34f}
(g,\epsilon^{g},\mu^{g})\circ (h,\epsilon^{h},\mu^{h})=
(gh,\epsilon^{gh},\mu^{gh})\ .
\end{equation}  The equality (as opposed to the additional data of a natural transformation) 
expresses the fact that the action of $G$ on $\bA$ is strict.

We now describe the category $\cVA$ 
explicitly.
\begin{enumerate}
\item The objects of $\cVA$ are pairs $(X,(M,\rho))$ of {objects} $X$ in $G\BC$ and $(M,\rho)$ in $\bV_{\bA}^{G}(X)$.
\item\label{gio43f34ferfef} A morphism $(f,\phi):(X,(M,\rho))\to (X^{\prime},(M^{\prime},\rho^{\prime}))$ consists of a morphism
$f:X\to X^{\prime}$ in $G\BC$ and a morphism $\phi:f_{*}(M,\rho)\to (M^{\prime},\rho^{\prime})$ in $\bV_{\bA}^{G}(X^{\prime})$.
\item The composition of morphisms is given by 
$$(f^{\prime},\phi^{\prime})\circ (f,\phi):=(f^{\prime}\circ f,\phi^{\prime}\circ f^{\prime}_{*}(\phi))\ .$$
\end{enumerate}

The functor $$\pi:\cVA\to G\BC\ , \quad  (X,(M,\rho))\mapsto X$$ is the obvious functor 
  which forgets the second component.

We now start with the description of the symmetric monoidal structure on $\cVA$.  
  
  Let $*$ denote the one-point space. 
Then we can  consider the  equivariant $*$-controlled $\bA$-object
$1_{*}=(M^{unit},\rho^{unit})$ in $\bV_{A}^{G}(*)$
defined as follows: 
\begin{enumerate}
\item The functor $M^{unit}:\cB(*)\to \bA$ is uniquely determined by
$M^{unit}(\{*\}):=1_{\bA}$.
\item $\rho^{unit}(g):=\epsilon^{g}$ for all $g$ in $G$ (see \ref{fweiufeiwofewfewf}.).
\end{enumerate}

\begin{ddd} \label{4gtiuo34ergg}The tensor unit of $\cVA$ is defined to be the object $1_{\cVA}:=(*,1_{*})$.
\end{ddd}

We now construct   the bifunctor \begin{equation}\label{23fb2jkbnfij23de23d23d}
-\otimes_{\cVA}- :\cVA\times \cVA\to \cVA\ .
\end{equation}
 We start with its definition on objects. We consider two objects
$(X,(M,\rho))$ and $(X^{\prime},(M^{\prime},\rho^{\prime}))$ in $\cVA$. 
Then we define the functor
$$M\boxtimes  M^{\prime}:\cB(X\otimes X^{\prime})\to \bA$$
as follows:
\begin{enumerate}
\item For every $B$ in $\cB(X\otimes X^{\prime})$ we set (see  Convention \ref{choosesum})
$$(M\boxtimes  M^{\prime})(B):=\bigoplus_{(x,x^{\prime})\in B} M(\{x\})\otimes_{\bA}M^{\prime}(\{x^{\prime}\})\ .$$
Note that the sum has finitely many non-zero summands because of Definition \ref{def:Xcontrolledobject} (\ref{def:Xcontrolledobject:it4}).
\item If $B^{\prime}$ is in $\cB(X\otimes X^{\prime})$ such that $B^{\prime}\subseteq B$, then the morphism
$$(M\boxtimes  M^{\prime})(B^{\prime}\subseteq B)\colon(M\boxtimes  M^{\prime})(B^{\prime})\to 
(M\boxtimes  M^{\prime})(B)$$
is given by the canonical map
$$\bigoplus_{(x,x^{\prime})\in B^{\prime}} 
M(\{x\})\otimes_{\bA}M^{\prime}(\{x^{\prime}\})
\to \bigoplus_{(x,x^{\prime})\in B } 
M(\{x\})\otimes_{\bA}M^{\prime}(\{x^{\prime}\})$$
as described in Convention \ref{choosesum}.
\end{enumerate}
 By using our Convention \ref{choosesum} and the universal property of the direct sum, one easily checks that this describes a functor satisfying the first three conditions of Definition \ref{def:Xcontrolledobject}. 
 
 We now define the family
 $\rho\boxtimes \rho^{\prime}$
 as follows:
 $$\hspace{-1cm}(\rho\boxtimes \rho^{\prime})(g)_{B}:= \oplus_{(x,x^{\prime})\in B} \rho(g)_{x}\otimes \rho^{\prime}(g)_{x^{\prime}}:
  \bigoplus_{(x,x^{\prime})\in B } 
M(\{x\})\otimes_{\bA}M^{\prime}(\{x^{\prime}\})\to \bigoplus_{(x,x^{\prime})\in B } 
gM(\{g^{-1}x\})\otimes_{\bA}gM^{\prime}(\{g^{-1}x^{\prime}\})$$
using the notation \eqref{v4rr4vhu34fr3f}. 
One checks using \eqref{rgkjbkrjg34f34f34f}  that
$(M\boxtimes M^{\prime},\rho\boxtimes \rho^{\prime})$  satisfies the remaining condition of  Definition \ref{def:Xcontrolledobject} 
 and therefore belongs to $\bV_{\bA}^{G}(X\otimes X^{\prime})$.

 \begin{ddd}
We define the bifunctor \eqref{23fb2jkbnfij23de23d23d} on objects by
$$(X,(M,\rho))\otimes_{\cVA}(X^{\prime},(M^{\prime},\rho^{\prime})):=(X\otimes X^{\prime},(M\boxtimes M^{\prime},\rho\boxtimes \rho^{\prime}))\ .$$
 \end{ddd}

Let $(f,\phi):(X_{0},(M_{0},\rho_{0}))\to (X_{1},(M_{1},\rho_{1}))$ be a morphism in $\cVA$ (see \ref{gio43f34ferfef}). 
Then we define the morphism
$$\hspace{-1cm}(g,\psi):=(f,\phi)\otimes_{\cVA} (X^{\prime},(M^{\prime},\rho^{\prime})):  (X_{0}\otimes X^{\prime} ,(M_{0}\boxtimes M^{\prime}, \rho_{0}\boxtimes \rho^{\prime}))\to  (X_{1}\otimes X^{\prime} ,(M_{1}\boxtimes M^{\prime}, \rho_{1}\boxtimes \rho^{\prime}))$$ as follows.
\begin{enumerate}
\item We set $g:=f\otimes \id_{X^{\prime}}:X_{0}\otimes X^{\prime}\to X_{1}\otimes X^{\prime}$ using the tensor product in $G\BC$.
\item In order to describe the morphism
$$\psi:(f\otimes\id_{X^{\prime}})_{*} (M_{0}\boxtimes M^{\prime}, \rho_{0}\boxtimes \rho^{\prime})\to (M_{1}\boxtimes M^{\prime}, \rho_{1}\boxtimes \rho^{\prime})$$ we use Corollary \ref{fuewihfiuwehfewfewfwef}. We must describe the matrix
$$(\psi^{f\otimes \id_{X^{\prime}}}_{(x_{1},y^{\prime}),(x_{0},x^{\prime})})_{(x_{0},x^{\prime})\in X_{0}\times X^{\prime}, (x_{1},y^{\prime})\in X_{1}\times X^{\prime}}\ .$$
Now note that by definition 
$$(M_{0}\boxtimes M^{\prime})(\{x_{0},x^{\prime}\})\cong M_{{0}}(\{x_{0}\})\otimes_{\bA} M^{\prime}(\{x^{\prime}\})$$ so that we can set 
$$\psi^{f\otimes  \id_{X^{\prime}}}_{(x_{1},y^{\prime}),(x_{0},x^{\prime})}:=\phi^{f}_{x_{1},x_{0}}\otimes_{\bA} (\id_{(M^{\prime},\rho^{\prime})})_{y^{\prime},x^{\prime}} :M_{0}(\{x_{0}\})\otimes_{\bA} M^{\prime}(\{x^{\prime}\})\to M_{1}(\{x_{1}\})\otimes_{\bA} M^{\prime}(\{y^{\prime}\})\ .$$
 One easily checks that
this matrix satisfies the conditions listed in  Corollary \ref{fuewihfiuwehfewfewfwef}  and therefore represents the desired morphism.
\end{enumerate}
In a similar manner we define
$(X,(M,\rho))\otimes (f^{\prime},\phi^{\prime})$
for a morphism
 $(f^{\prime},\phi^{\prime}):(X^{\prime}_{0},(M^{\prime}_{0},\rho^{\prime}_{0}))\to (X^{\prime}_{1},(M^{\prime}_{1},\rho_{1}))$.
 

\begin{ddd}
We define the bifunctor \eqref{23fb2jkbnfij23de23d23d} 
on morphisms by the preceding description.
\end{ddd}

It is straightforward to check that \eqref{23fb2jkbnfij23de23d23d} is  a bifunctor, i.e., that its description on morphisms is compatible with composition. 

Next we define the associativity constraint
$\alpha^{\cVA}$. We consider three objects
$(X,(M,\rho))$, $(X^{\prime},(M^{\prime},\rho^{\prime}))$, and
$(X^{\prime\prime},(M^{\prime\prime},\rho^{\prime\prime}))$. Then
$$ (f,\phi):=
\alpha_{(X,(M,\rho)),(X^{\prime},(M^{\prime},\rho^{\prime})),(X^{\prime\prime},(M^{\prime\prime},\rho^{\prime\prime}))}$$ must be a  morphism
$$ 
((X\otimes X^{\prime})\otimes X^{\prime\prime},((M\boxtimes M^{\prime})\boxtimes M^{\prime\prime},(\rho\boxtimes \rho^{\prime})\boxtimes \rho''))\to (X\otimes (X^{\prime}\otimes X^{\prime\prime}),(M\boxtimes (M^{\prime}\boxtimes M^{\prime\prime}),\rho\boxtimes (\rho^{\prime}\boxtimes \rho'')))$$ 
We set
$$f:=\alpha_{X,X^{\prime},X^{\prime\prime}}$$
using the associativity constraint of $G\BC$.
The second component $\phi$ is given via Corollary \ref{fuewihfiuwehfewfewfwef}  by the matrix whose only non-trivial entries are 
$$\phi^{f}_{(x,(x^{\prime},x^{\prime\prime})), ((x,x^{\prime}),x^{\prime\prime})}:=\alpha_{M(\{x\}),M^{\prime}(\{x^{\prime}\}),M^{\prime\prime}(\{x''\})}$$
using the associativity constraint of $\bA$.
The first condition  of Corollary \ref{fuewihfiuwehfewfewfwef} is satisfied for the diagonal entourage of $X\times (X^{\prime}\times X'')$, and for the second condition we use that $G$ acts on $\bA$ by symmetric monoidal functors, in particular the first relation in Definition \ref{34f8924fregerg} for $\mu^{g}$ for all $g$ in $G$, see \ref{vuhiwefwefwefewfwef2r}.

\begin{ddd}
We define  the associativity constraint
$\alpha^{\cVA}$ by the description above.
\end{ddd}

It is straightforward but tedious to check that $\alpha^{\cVA}$ is a natural transformation.

Following Definition \ref{4gtiuo34ergg} the unit constraint $\eta^{\cVA}$  of $\cVA$ is implemented by morphisms
$$(f,\phi): (*\otimes X,  (M^{unit}\boxtimes M,\rho^{unit}\boxtimes \rho))\to (X,(M,\rho))$$ for all objects $(X,(M,\rho))$ of $\cVA$. We set
$$f:=\eta_{X}$$
using the unit constraint of $G\BC$. Note that
$$(M^{unit}\boxtimes M)(\{(*,x)\})\cong 1_{\bA}\otimes_{\bA} M(\{x\})\ .$$
Hence, using  Corollary \ref{fuewihfiuwehfewfewfwef}, we can define  morphism $\phi$ such that the non-trivial entries of its matrix are
$$\phi^{f}_{x,(*,x)}:=\eta_{M(\{x\})}$$
using the unit constraint of $\bA$. It is easy to check that this matrix satisfies the first condition of Corollary \ref{fuewihfiuwehfewfewfwef} for the diagonal of $X$
and the second condition since the morphisms $\epsilon^{g}$ in \ref{fweiufeiwofewfewf} satisfy the relation {of Definition }\ref{34f8924fregerg}.\ref{fiowfewef23rr3} for all $g$ in $G$.

\begin{ddd}
We define the unit constraint $\eta^{\cVA}$ {by} the description above.
\end{ddd}

It is straightforward to check that $\eta^{\cVA}$ 
is a natural transformation.

Finally we define the symmetry constraint $\sigma^{\cVA}$. 
We consider two objects $(X,(M,\rho))$ and $(X^{\prime},(M^{\prime},\rho^{\prime}))$ of $\cVA$.
Then we must define a morphism
$$(f,\phi):(X\otimes X^{\prime},(M\boxtimes M^{\prime},\rho\boxtimes \rho^{\prime}))\to 
(X^{\prime}\otimes X ,(M^{\prime}\boxtimes M ,\rho^{\prime}\boxtimes \rho ))\ .$$
We set
$$f:=\sigma_{X,X^{\prime}}$$ using the symmetry contraint for $G\BC$. The morphism
$\phi$ is the given, using Corollary \ref{fuewihfiuwehfewfewfwef}, by the matrix whose only non-trivial entries are
$$\phi^{f}_{(x^{\prime},x),(x,x^{\prime})}:=\sigma_{M(\{x\}),M^{\prime}(\{x^{\prime}\})}$$
using the symmetry constraint of $\bA$.
One easily {checks} that the first condition of 
 Corollary \ref{fuewihfiuwehfewfewfwef} is satisfied for the diagonal entourage of $X^{\prime}\times X$. In order to verify the second condition we use that
 the transformations $\mu^{g}$ in \ref{vuhiwefwefwefewfwef2r}
 satisfy {Definition} \ref{34f8924fregerg}.\ref{evfiwehfioewfewfewfewff} for every $g$ in $G$.

\begin{ddd}
We define the symmetry constraint  $\sigma^{\cVA}$  of $\cVA$ by the description above.
\end{ddd}

It is straightforward to check that $\sigma^{\cVA}$ 
is a natural transformation.

\begin{prop}\label{efiweufowefewfw}
The functor $-\otimes_{\cVA}-$ and the object $1_{\cVA}$ together with the natural isomorphisms
$\alpha^{\cVA}$, $\eta^{\cVA}$ and $\sigma^{\cVA}$
define a symmetric monoidal structure on $\cVA$. 

The functor
$\pi:\cVA\to G\BC$ preserves the tensor product and   the tensor unit as well as the associator, unit, and symmetry transformations.
 \end{prop}
 \begin{proof}
 One verifies the relations listed in Definition \ref{gui23r32reger} in a straightforward manner by inserting the definitions and using that the corresponding relations are satisfied for the symmetric monoidal structures on $\bA$ and $G\BC$. 
 \end{proof}

Let $X$ and $X^{\prime}$ be $G$-bornological coarse spaces.
\begin{prop}\label{efiweufowefewfw1}
The functor
$$\boxtimes_{X,X^{\prime}}:\bV^{G}_{\bA}(X)\times \bV^{G}_{\bA}(X^{\prime})\to \bV_{\bA}^{G}(X{\otimes} X^{\prime})$$ obtained in \eqref{f3foi34jfio34jf34f34f3f} is additive in both variables.
\end{prop}
\begin{proof}
Let $(M_{i},\rho_{i})$ be in $\bV_{\bA}^{G}(X)$ for $i=0,1$  and $(M^{\prime},\rho)$ be in $\bV_{\bA}^{G}(X^{\prime})$.  In view of the symmetry it suffices to show that the canonical morphism
$$(M_{0}\boxtimes_{X,X^{\prime}}M^{\prime})\oplus (M_{1}\boxtimes_{X,X^{\prime}}M^{\prime})\to (M_{0}\oplus M_{1})\boxtimes_{X,X^{\prime}} M^{\prime}$$
is an isomorphism. In view of Conditions 
 \ref{def:Xcontrolledobject}.\ref{def:Xcontrolledobject:it1} and \ref{def:Xcontrolledobject}.\ref{def:Xcontrolledobject:it3} it suffices to show that
$$\left[(M_{0}\boxtimes_{X,X^{\prime}}M^{\prime} )\oplus (M_{1}\boxtimes_{X,X^{\prime}}M^{\prime})\right](\{(x,x^{\prime})\})\to \left[(M_{0}\oplus M_{1})\boxtimes_{X,X^{\prime}} M^{\prime}\right](\{(x,x^{\prime})\})$$
is an isomorphism for every point $(x,x^{\prime})$ in $X\times X^{\prime}$.
By inserting the definitions we see that  this morphism is the same as
$$(M_{0}(\{x\})\otimes_{\bA} M^{\prime}(\{x^{\prime}\}))\oplus (M_{1}(\{x\})\otimes_{\bA} M^{\prime}(\{x^{\prime}\}))\to (M_{0}(\{x\})\oplus M_{1}(\{x\}))\otimes_{\bA} M^{\prime}(\{x^{\prime}\})\ .$$
 But this last morphism is an isomorphism since the tensor product in $\bA$ is additive in the first argument.
\end{proof}

Let $f\colon X\to X^{\prime}$ and $f^{\prime}\colon Y\to Y^{\prime}$ be two morphisms of $G$-bornolgical coarse spaces. Let $(M,\rho)$  be in $\bV_{\bA}^{G}(X)$
and $(N,\eta)$ be in $\bV_{A}^{G}(Y)$.  

\begin{lem}\label{lemmaiso} The morphism   	 	$$(f\otimes f^{\prime})_{*}((M,\rho)\boxtimes_{X ,Y} (N,\eta)) \to  f_{*}(M,\rho)\boxtimes_{X^{\prime},Y^{\prime}} f^{\prime}_{*}(N,\eta)$$  in $\bV_{\bA}^{G}(X^{\prime}\otimes Y^{\prime})$ (see {\eqref{vwrkjbwwejkvjbnkwevewvw}})
 is an isomorphism.
\end{lem}
\begin{proof}
In view of Conditions 
 \ref{def:Xcontrolledobject}.\ref{def:Xcontrolledobject:it1} and \ref{def:Xcontrolledobject}.\ref{def:Xcontrolledobject:it3} it suffices to show that
 $$[(f\otimes f^{\prime})_{*}((M,\rho)\boxtimes_{X,Y} (N,\eta))](\{(x^{\prime},y^{\prime})\}) \to  [f_{*}(M,\rho)\boxtimes_{X^{\prime},Y^{\prime}} f^{\prime}_{*}(N,\eta)](\{(x^{\prime},y^{\prime})\})$$
 is an isomorphism for every point $(x^{\prime},y^{\prime})$ in $X^{\prime}\times Y^{\prime}$. Inserting the definitions this morphism is given by
 \begin{equation}\label{fehi23ufhwefwfefe}\bigoplus_{(x,y)\in (f\times f^{\prime})^{-1}(\{x^{\prime},y^{\prime}\})} M(\{x\})\otimes_{\bA} N(\{y\})\to  \left(\bigoplus_{x\in f^{-1}(\{x^{\prime}\})} M(\{x\})\right) \otimes_{\bA} \left(\bigoplus_{y\in (f^{\prime})^{-1}(\{y^{\prime}\})} N(\{y\})\right)\end{equation}
 which for every $(x,y)$ in  $(f\times f^{\prime})^{-1}(\{x^{\prime},y^{\prime}\})$
 is the morphism
  $$  M(\{x\})\otimes_{\bA} N(\{y\})\to  \left(\bigoplus_{x\in f^{-1}(\{x^{\prime}\})} M(\{x\})\right) \otimes_{\bA} \left(\bigoplus_{y\in (f^{\prime})^{-1}(\{y^{\prime}\})} N(\{y\})\right) $$
  induced by the inclusions of the respective summands of the tensor factors.
  Since the tensor product in $\bA$ preserves sums in both arguments we conclude that 
 \eqref{fehi23ufhwefwfefe} is an isomorphism.
\end{proof}


In view of Theorem \ref{rgio34t34r34rerg} the  Propositions \ref{efiweufowefewfw} and \ref{efiweufowefewfw1} and Lemma \ref{lemmaiso} now imply:

\begin{theorem}\label{mainthm}
If $\bA$ is a symmetric monoidal  additive category with a strict action of $G$ by symmetric monoidal functors, then the functor $loc\circ \bV_{\bA}^{G}:G\BC\to  \Add_{\infty}$ admits a refinement to a {lax} symmetric monoidal functor
$$\bV_{\bA}^{G,\otimes}:\Nerve(G\BC^{\otimes})\to \Add_{\infty}^{\otimes}\ .$$
\end{theorem}

\subsection{The symmetric monoidal $K$-theory functor for additive categories}\label{foiwejfowfwefefwef}

\newcommand{\dgCat}{\mathbf{dgCat}}

In \cite{buci} a universal $K$-theory functor $$\UK:\Nerve(\Add_{1})\to \cM_{loc}$$
was considered, where $\cM_{loc}$ is the category of non-commutative motives of Blumberg-Gepner-Tabuada \cite{MR3070515}. This functor was defined {as the upper horizontal composition  in the diagram}
$$\xymatrix{\Nerve(\Add_{1})\ar[d]^{loc}
\ar[r]^{\Ch^{b}[-]_{\infty}}&\Cat^{ex}_{\infty}\ar[r]^{\cU_{loc}}& \cM_{loc}\\\Add_{\infty}\ar@{..>}[urr]_{\UK_{\infty}}&&
}\ ,$$
where $\cU_{loc}$ is the universal localizing invariant, and $\Ch^{b}[-]_{\infty}$ sends an additive category $\bA$ to the stable $\infty$-category of bounded chain complexes over $\bA$ with homotopy equivalences   inverted. 
Since the functor $\UK$ preserves equivalences of additive categories we have the indicated factorization ${\UK_{\infty}}$.

  \begin{theorem}\label{rgreioghjo34tergergeg21}
The functor $\UK_{\infty}$ admits a symmetric monoidal refinement
$$\UK_{\infty}^{\otimes}:\Add_{\infty}^{\otimes}\to \cM_{loc}^{\otimes}\ .$$
\end{theorem}
\begin{proof}
 The proof of this theorem will be finished at the end of the present section.
 As a first step we observe that it suffices to construct a symmetric monoidal refinement  $\UK^{\otimes}$ of $\UK$.
 Then we obtain the symmetric monoidal refinement  $\UK^{\otimes}_{\infty} $ of $\UK_{\infty}$ from the universal property of the symmetric monoidal localization $loc^{\otimes}:\Nerve(\Add_{1}^{\otimes})\to \Add_{\infty}^{\otimes}$ using \cite{hinich}.

By   \cite{bgt-2} the universal localizing invariant $\cU_{loc}$  refines to a symmetric monoidal functor
$$\cU^{\otimes}_{loc}:\Cat^{ex,\otimes}_{\infty}\to \cM^{\otimes}_{loc}\ .$$  It therefore remains to produce a symmetric monoidal functor $$St^{\otimes}: \Add^{\otimes}_{\infty} \to \Cat^{ex,\otimes}_{\infty}$$ refining $\Ch^{b}[-]_{\infty}$. We use the symbol $St$ in order to indicate that this functor is related  with stabilization.

We are going to use the following notation.
 The category $\dgCat_{1}$ is the $1$-category of {small} dg-categories. The  set  $W_{Morita}$ is  the set of Morita equivalences, i.e., functors between $dg$-categories $\cC\to \cD$ which induce an equivalence of derived categories \cite[Sec. 4.6]{MR2275593},  \cite[Def. 2.29]{cohn}.

 The category $\dgCat_{1}$  contains the full subcategory
$\dgCat_{1,flat}$  of locally flat dg-categories, i.e., dg-categories $\bC$ with the property that for every two objects $C,C^{\prime}$ in $\bC$ the complex $\Hom_{\bC}(C,C^{\prime})$ consists of flat $\Z$-modules.  It furthermore contains the full subcategory of pre-triangulated $dg$-categories \cite[Sec. 4.5]{MR2275593}, \cite{BKdg}. 

Furthermore, $\Cat^{ex}_{\infty,H\Z}$ is the category of $H\Z$-linear stable {idempotent complete} $\infty$-categories and $H\Z$-linear exact functors, and $\cF$ forgets the $H\Z$-linear structure.  For the equivalence  marked by $DK$ (for Dold-Kan) we refer to \cite{cohn}. 
\begin{prop}
We have the  bold part of the following commuting diagram:  \begin{equation}\label{fhwjkefkj23r23r23}
\xymatrix{&\Nerve(\dgCat_{1}^{pre})\ar[d]^{\subseteq}\ar[r]^{\ell}&\Nerve(\dgCat_{1}^{pre})[W^{-1}_{Morita}]\ar[d]^{\simeq}_{\iota} \ar[drr]^{\Nerve_{\infty}^{dg}(-)}&&\\\Nerve(\Add_{1})\ar@/^4cm/@{..>}[rrrr]^{St}\ar@{..>}[dr]^{\bQ}\ar[r]&\Nerve(\dgCat_{1})\ar[ur]^{\ell\circ \Ch^{b}}\ar[r]^{\ell}&\Nerve(\dgCat_{1})[W^{-1}_{Morita}]\ar[r]_{\simeq}^{DK}& \Cat_{\infty,H\Z}^{ex}\ar[r]_{\cF}&\Cat^{ex}_{\infty}\\ &\Nerve(\dgCat_{1,flat})\ar[u]\ar[r]&\Nerve(\dgCat_{1,flat})[W^{-1}_{Morita}]\ar[ur]\ar[u]_{!!}^{\simeq}& &}\ .\end{equation}
\end{prop}
\begin{proof}
\begin{enumerate}
\item For every $dg$-category the canonical inclusion  $\cC\to \Ch^{b}(\cC)$  represents the pretriangulated hull \cite[Sec. 4.5]{MR2275593}, \cite{BKdg}. In particular, the functor $\Ch^{b}$  has values in pretriangulated $dg$-categories.  
\item The two triangles in the corresponding square commute since 
  for every $dg$-category $\cC$ the canonical inclusion induces a Morita equivalence $\cC\to \Ch^{b}(\cC)$. To this end we
   use that the inclusion of $\cC$ into its triangulated hull is a Morita {equivalence} {\cite[Sec. 4.6]{MR2275593}}.  
%
  
\item The $dg$-nerve $\Nerve^{dg} :\Nerve(\dgCat_{1}^{pre})\to \Cat^{ex}_{\infty}$ preserves Morita equivalences and therefore descends to $\Nerve^{dg}_{\infty}$ as indicated.
\item We have an equivalence $\cF\circ DK \circ  \iota\simeq \Nerve_{\infty}^{dg}$, \cite[Prop. 3.3.2]{MR3607208}, see also \cite{Tabuada_2010}, \cite{toen}.
In order to provide more details we consider   the functor $Z^{0}:\dgCat_{1}^{pre}\to \Cat_{\infty}$   which associates to a $dg$-category  its underlying category (with $\Hom_{Z^{0}(\cC)}(A,B)=Z^{0}(\Hom_{\cC}(A,B))$) considered as an $\infty$-category. We  furthermore let $W_{\cC}$ be the morphisms in $Z^{0}(\cC)$ which become isomorphisms in the homotopy category $H^{0}(\cC)$ (with $\Hom_{H^{0}(\cC)}(A,B)=H^{0}(\Hom_{\cC}(A,B))$).
 Then both functors
 $$Z^{0}(\cC)\to N^{dg}(\cC) \ ,\quad 
 Z^{0}(\cC)\to \cF(DK(\iota(\cC)))$$ present 
 the localization $ Z^{0}(\cC)\to Z^{0}(\cC)[W_{\cC}^{-1}]$.
 
\end{enumerate}

\end{proof}

  The horizontal composition given by the  middle row  in \eqref{fhwjkefkj23r23r23}  defines a  functor $St$.

\begin{lem}
The functor $St$ is equivalent to the functor
$\Ch^{b}(-)_{\infty}$ constructed in \cite[Prop. 2.11]{buci}.
\end{lem}
\begin{proof}
By \cite[Rem. 2.9]{buci} we have the first equivalence of functors in the chain 
$$\Ch^{b}(-)_{\infty}\simeq \Nerve^{dg}\circ \Ch^{b}(-)\simeq \Nerve^{dg}_{\infty}\circ \ell\circ \Ch^{b}$$ from $\Add_{1}$ to $\Cat^{ex}_{\infty}$. This implies the Lemma in view of the commutativity of 
\eqref{fhwjkefkj23r23r23}.
%
%
%
\end{proof}

\begin{prop} The functor $St$ has a symmetric monoidal refinement $St^{\otimes}$.
\end{prop}
\begin{proof}
All   $1$-categories in the lower {two} lines of the diagram \eqref{fhwjkefkj23r23r23} have symmetric monoidal structures and the functors
connecting them have canonical symmetric monoidal refinements. The same is true for the $\infty$-categories and the remaining functors except  for
$\Nerve(\dgCat_{1})[W^{-1}_{Morita}]$ and the corresponding functors.
The problem is that the tensor product of dg-categories is not compatible with Morita equivalences and therefore does not descend to the localization directly. For this reason one considers the
subcategory of locally flat dg-categories and uses the equivalence $!!$ in order to transfer the symmetric monoidal structures. So in order to construct the symmetric monoidal refinement of 
$St$ we must bypass this node of the diagram. To this end we use a symmetric monoidal flat resolution functor $\bQ$ as indicated. The left triangle in  \eqref{fhwjkefkj23r23r23}  is filled by a natural transformation (not an isomorphism), but
 the square
$$\xymatrix{N(\Add_{1})\ar[rr]\ar[dr]^{\bQ}&&\Nerve(\dgCat_{1})[W^{-1}_{Morita}]\\&\Nerve(\dgCat_{1,flat})\ar[r]&\Nerve(\dgCat_{1,flat})[W^{-1}_{Morita}]\ar[u]^{\simeq}_{!!}
}$$ {does
commute}. We then get the following commuting diagram of symmetric monoidal functors
 $$\xymatrix{  \Nerve(\Add_{1}^{\otimes})\ar@/^1.5cm/[rrrr]^{St^{\otimes}}\ar[dr]^{\bQ^{\otimes}}&&& \Cat_{\infty,H\Z}^{ex,\otimes}\ar[r]^{\cF}&\Cat^{ex,\otimes}_{\infty}\\&\Nerve(\dgCat_{1,flat}^{\otimes})\ar[r]&\Nerve(\dgCat_{1,flat})^{\otimes}[W^{-1}_{Morita}]\ar[ur]& &}$$
defining the symmetric monoidal refinement $St^{\otimes}$ of $St$.

It remains to argue that a  symmetric monoidal  flat resolution functor $\bQ$ exists.
We start with the following well-known fact. 
\begin{lem}\label{giojwggrgreg}
\begin{enumerate}
\item There exists a functor $Q$ fitting into the commuting diagram 
$$\xymatrix{&\Ch_{flat}\ar[d]\\\Ab\ar[ur]^{Q}\ar[r]^{(-)[0]}&\Ch}$$
such that the filler is a quasi-isomorphism.
\item The functor $Q$ has a lax symmetric monoidal structure.
\end{enumerate}
\end{lem}
\begin{proof} The natural idea works. The functor $Q$ sends   $A$ in $\Ab$ to $$Q(A):=(F_{1}(A)\stackrel{d_{A}}{\to} F_{0}(A))$$ in $\Ch$, where $F_{0}(A):=\Z[A]$ is the free abelian group generated by the underlying set of $A$, $F_{1}(A)$ is the kernel of the canonical homomorphism $F_{0}(A)\to A$, and $d_{A}$ is the inclusion.
\end{proof}

We define the flat resolution functor for additive categories by 
$$\xymatrix{&&\dgCat_{1,flat}\ar[d]\\\Add_{1}\ar[urr]^{\bQ}\ar[r]&\Cat_{\Ab}\ar[r]\ar@{.>}[ur]&\dgCat_{1}}$$
where the dotted arrow is the natural functor induced from the lax symmetric monoidal functor $Q$  
which provides a functor from $\Ab$-enriched categories to $\Ch$-enriched categories with flat $\Hom$-complexes. Furthermore, the symmetric monoidal structure on $Q$ induces naturally a symmetric monoidal structure on $\bQ$.
\end{proof}

This finishes the proof of Theorem \ref{rgreioghjo34tergergeg21}.
\end{proof}

\bibliographystyle{alpha}
\bibliography{symcontr}

\end{document}